\documentclass[12pt,english]{amsart}
\usepackage{ulem} 
\calclayout
\usepackage{stmaryrd}
\usepackage{graphicx, epsfig, psfrag}
\usepackage{epstopdf}
\usepackage{color}
\usepackage{enumerate}
\usepackage{mathrsfs}%

\newtheorem{rem}{Remark}
\theoremstyle{plain}
\newtheorem{theorem}{Theorem}
\newtheorem{lemma}[theorem]{Lemma}
\newtheorem{corollary}[theorem]{Corollary}

\newcommand\A{\mathcal{A}}
\newcommand\HH{{\mathscr{H}}}

\newcommand\T{\mathcal T}
\newcommand\Har{\mathbb H}
\newcommand\RR{\mathbb R}

\newcommand\NNN{\mathcal N}
\newcommand\NNNE{{\mathcal{N}}_e}
\newcommand\NNNEc{{\mathcal{N}}_{\partial \tau \backslash e}}

\newcommand\bg{{g}}

\newcommand\xx{{\boldsymbol x}}

\newcommand\vv{{\boldsymbol v}}
\newcommand\I{\mathcal{I}}

\newcommand{\trih}{\mathcal{T}_h}
\newcommand{\triH}{{\mathcal{T}_H}}

\newcommand{\btriH}{{\partial\mathcal{T}_H}}

\newcommand\dO{\partial\Omega}
\newcommand\dtau{{\partial\tau}}
\newcommand\OO{\Omega}

\newcommand\amin{a_{\text{min}}}
\newcommand\amax{a_{\text{max}}}
\newcommand\rhomin{\rho_{\text{min}}}
\newcommand\rhomax{\rho_{\text{max}}}
\newcommand\tLambda{{\widetilde\Lambda}_h}
\newcommand\Lambdams{{\Lambda_h^{\rm{ms}}}}

\newcommand\tlambda{\widetilde\lambda_h}

\newcommand\tmu{{\widetilde\mu}_h}

\newcommand\tnu{{\widetilde\nu}_h}
\def\div{\operatorname{div}}

\def\bgrad{\operatorname{\boldsymbol{\operatorname{\nabla}}}} 
\def\span{\operatorname{span}}

\definecolor{bluegreen}{rgb}{0,0.75,0.75}

\usepackage{lipsum}
\usepackage{amsfonts}
\usepackage{graphicx}
\usepackage{epstopdf}
\usepackage{algorithmic}
\ifpdf
  \DeclareGraphicsExtensions{.eps,.pdf,.png,.jpg}
\else
  \DeclareGraphicsExtensions{.eps}
\fi

\begin{document}
\title[Spectral ACMS]{Spectral ACMS: A robust localized Approximated Component Mode Synthesis Method}
\date{May 19, 2025}
\author{Alexandre L. Madureira}
\address{
Laborat\'orio Nacional de Computa\c c\~ao Cient\'\i fica, Petr\'opolis - RJ, Brazil} 
\address{
Funda\c c\~ao Get\'ulio Vargas, Rio de Janeiro - RJ, Brazil} 
\email{alm@lncc.br, alexandre.madureira@fgv.br}
\author{Marcus Sarkis}
\address{
Mathematical Sciences Department
Worcester Polytechnic Institute, USA}%
\email{msarkis@wpi.edu}

\title{Spectral ACMS: A robust localized Approximated Component Mode Synthesis Method}
  
\thanks{This material is based upon work supported by the National Science Foundation under Grant No. DMS-1929284 while the authors were in residence at the Institute for Computational and Experimental Research in Mathematics in Providence, RI, during the semester program ``Numerical PDEs: Analysis, Algorithms, and Data Challenges.''  Some of the results were stated without proofs in the conference paper~\cite{zbMATH07621527}.  The first author was supported by the FAPERJ grants E-26/210.128/2022, E-26/211.143.2021, E-26/210.162/2019. The second author was supported by National Science Foundation (Grant No. NSF-MPS 1522663). }

\begin{abstract}
We consider finite element methods of multiscale type to approximate solutions for two-dimensional symmetric elliptic partial differential equations with heterogeneous $L^\infty$ coefficients. The methods are of Galerkin type and follow the Variational Multiscale and Localized Orthogonal Decomposition--LOD approaches in the sense that it decouples  spaces into \emph{multiscale} and \emph{fine} subspaces. In a first method, the multiscale basis functions are obtained by mapping coarse basis functions, based on corners used on primal iterative substructuring methods, to functions of global minimal energy. This approach delivers quasi-optimal a priori error energy approximation with respect to the mesh size, but it is not robust with respect to high-contrast coefficients. In a second method, edge modes based on local generalized eigenvalue problems are added to the corner modes. As a result, optimal a priori error energy estimate is achieved which is mesh and contrast independent. The methods converge at optimal rate even if the solution has minimum regularity, belonging only to the Sobolev space $H^1$. 
\end{abstract}

\maketitle

\section{Introduction}
Let $u:\OO\to\RR$ be the weak solution of 
\begin{equation}\label{e:elliptic}
\begin{gathered}
-\div\big(\A\bgrad u\big)=f\quad\text{in }\OO,
\\
u=0\quad\text{on }\partial\OO,
\end{gathered}
\end{equation}
where $\OO\subset\RR^2$, and is an open bounded domain with polygonal boundary $\partial\Omega$, the symmetric tensor $\A\in[L^\infty(\OO)]_{\text{sym}}^{2\times2}$ is uniformly positive definite almost everywhere, and $f\in L^2(\Omega)$ is given. For almost all $\xx\in\OO$ let the positive constants $\amin$ and $\amax$ be such that 
\begin{equation}\label{e:bounds}
  \amin|\vv|^2
  \le a_-(\xx)|\vv|^2\le\A(\xx)\,\vv\cdot\vv
  \le a_+(\xx)|\vv|^2\le\amax|\vv|^2
\quad\text{for all }\vv\in\RR^2,
\end{equation}
where $a_-(\xx)$ and $a_+(\xx)$ are the smallest and largest eigenvalues of 
$\A(\xx)$. Let $\rho\in L^\infty(\Omega)$\label{p:rhodef}
be chosen by the user and such that $\rho(\xx)\in[\rhomin,\rhomax]$ almost everywhere for some positive constants $\rhomin$ and $\rhomax$. Consider $\bg$ such that
\[
f=\rho\bg,\label{p:bg}
\]
and then the $\rho$-weighted $L^2(\Omega)$ norm\label{p:rholtnorm} $\|\bg\|_{L_\rho^2(\Omega)}:=\|\rho^{1/2}\bg\|_{L^2(\Omega)}=\|f\|_{L_{1/\rho}^2(\Omega)}$ is finite. The introduction of the weight $\rho$ is to balance $u$ and $f$ with respect to the tensor $\A$, adding flexibility to the method and making error estimates more meaningful. It allows for a \emph{fairer measure} of the error when the right hand side of error estimates depends on $f$ only. Also, for high-contrast problems, it might compensate for local low coercivity of $\A$; see the end of Section~\ref{s:highcase} for more details. We note that a related numerical work, without any proofs,  was presented in the conference paper~\cite{zbMATH07621527}.  Here the main goal is the corresponding analysis as sharp as possible without any hidden constants.

 For $v$, $w\in H^1(\Omega)$ let 
\[
a(v,w)=\int_\Omega\A \bgrad v\cdot\bgrad w\,d\xx,
\]
and denote by $(\cdot,\cdot)$ the $L^2(\Omega)$ inner product.

Although  in general the solution $u$ of~\eqref{e:elliptic} only belongs to the
Sobolev space $H^1(\Omega)$, a priori error analyses of multiscale methods established on the literature often rely on solution regularity; see~\cite{MR1286212,MR2684351,MR1979846,MR2058933,MR3109775,HMV,MR3225627,MR3704855,MR1642758,MR2383203,MR2030161} and references therein.

Considering the low contrast case, some methods require minimum regularity, as the generalized finite element methods~\cite{MR2801210}, the rough polyharmonic splines~\cite{MR3177856}, the variational multiscale method~\cite{MR1660141,MR2300286}, the Localized Orthogonal Decomposition (LOD)~\cite{MR4191211,MR3246801,MR2831590,MR3591945} and the oversampling Multiscale Finite Element Method~\cite{MR3123820}. The general idea is to decompose the solution spaces as a direct sum of \emph{fine} (local) and \emph{multiscale} (low dimensional, nonlocal) spaces. The final approximate solution belongs to the multiscale space. The LOD approximation~\cite{hellman17_mms,MR3552482} also  works for the high contrast cases when the local Poincar\'e inequality is not large; see Remark~\ref{r:rem1}.

However, there are several domain decomposition solvers that are optimal with respect to mesh and contrast, relying on coarse basis functions from local generalized eigenvalue problems. The \emph{adaptive choice of primal constraints} method was introduced to ensure robustness with respect to contrast for non-overlapping domain decomposition methods based on FETI-DP~\cite{MR1802366,MR1921914} and BDDC~\cite{MR2047204}. References~\cite{MR3612901,MR3582898,DP2013,MR3678572,MR3350292,MR3303686,MR3546980,MR2030628,MR2277024,MR2334131,ZAMPINI,MR3089678} elaborate on this approach. For overlapping domain decomposition see~\cite{MR3033238,MR2916377,MR2718268,MR3350292,MR3175183}. Some of this ideas were incorporated in~\cite{CHUNG2018298,Chung2018} to obtain discretizations that depend only logarithmically on the contrast.

In~\cite{MR4238010} we introduced the Localized Spectral
Decomposition--LSD method for mixed and hybrid-primal methods~\cite{MR0431752}, that is,
we re-frame the LOD version in~\cite{MR3591945} into
the non-overlapping domain decomposition framework, and consider the
Multiscale Hybrid Method--MHM~\cite{AHPV,HMV,HPV}, which falls in the
BDDC and FETI-DP classes, and then explore adaptive choice of primal
constraints to generate the multiscale basis functions. We obtain
a discretization that is robust with respect to contrast.

In this paper we propose an \emph{Approximated Component Mode Synthesis}--ACMS type  method; see~\cite{MR1160133,MR1069651,craig1968coupling,MR3350765,MR3225627,MR2666649,hurty1960vibrations}.  The motivation is that known convergence results require extra regularity and are not uniform with respect to contrast~\cite{MR3225627}. The goal here is to develop a discretization that has optimal and robust a priori error approximation, assuming minimum regularity on the solution, $\A$ and $\rho$. 

To consider the LOD approach with Galerkin-Ritz projection, we use conforming primal iterative substructuring techniques~\cite{MR764237,MR842125,MR1367653,MR1302680,MR1113145,MR1109101,MR1857663,MR1469678,MR2104179} rather than BDDC and FETI-DP methods. Two versions are under consideration here, both of Galerkin type and based on edges and local harmonic extensions. The first method is simpler and converges at quasi-optimal rates, even under minimal regularity of the solution. We note, however, this method has a weak singularity at the coarse nodes and its properties deteriorate if the contrast of the coefficients increases. To circumvent these two issues, we modify the method by incorporating solutions of specially designed local eigenfunction problems, yielding optimal convergence rate uniformly with respect to contrast. 

The remainder of the this paper is organized as follows. Section~\ref{s:lod}
describes the substructuring decomposition into interior and interface unknowns, while our methods for low and high contrast coefficients are considered in Sections~\ref{s:lowcase} and~\ref{s:highcase}, respectively. In Section~\ref{s:localproblems}  we consider how to deal with local, elementwise problems. Numerical tests and some of the results of this paper, for the case of high-contrast only, were presented without proofs in~\cite{zbMATH07621527}. 

\section{Substructuring Formulation} \label{s:lod}
Let $\triH$ be a finite element regular partition of $\Omega$ based on triangles, with elements of characteristic length $H>0$. We denote the mesh skeleton by $\btriH$, and denote by $\NNN_H$ the set of nodes on $\btriH\backslash\dO$. For $h<H$, let $\trih$ be a refinement of $\triH$, in the sense that every (coarse) edge in $\btriH$ is a union of edges of elements in $\trih$. Let $\NNN_h$ be the set of nodes of $\trih$ on the skeleton $\btriH\backslash\dO$. Therefore, all nodes in $\NNN_h$ belong to edges of elements in $\triH$.

For $v\in H^1(\Omega)$ and a given set of elements $\T\subset\triH$, let 
\begin{equation*}
|v|_{H_\A^1(\Omega)}^2=\|\A^{1/2}\bgrad v\|_{L^2(\Omega)}^2,
\qquad
|v|_{H_\A^1(\T)}^2=\sum_{\tau\in\T}\|\A^{1/2}\bgrad v\|_{L^2(\tau)}^2. 
\end{equation*}
Let $V_h\subset H_0^1(\Omega)$ be the space of continuous piecewise linear functions associated with the fine mesh $\trih$. For the sake of reference, let $u_h\in V_h$ such that
\[
a(u_h,v_h)=(\rho g,v_h)\quad\text{for all }v_h\in V_h. 
\]
We assume that $u_h$ approximates $u$ well. Our numerical schemes yield good approximations for $u_h$ without ever computing it. As in~\cite{MR3123820}, we use such solution $u_h$ as a reference to estimate the error $u_h - u_h^{\rm{ms}}$. That is a departure from classical results where the error $u - u^{\rm{ms}}$ is first estimated. Since $u^{\rm{ms}}$ requires solving infinite dimensional problems, the $u^{\rm{ms}}$ is replaced by $u_h^{\rm{ms}}$ and the error $u^{\rm{ms}} - u_h^{\rm{ms}}$ is then estimated and which requires again density arguments. In summary, our approach analyze what it is implemented, the analysis is simpler and provides sharper estimates  since we just need to estimate the error $u_h - u_h^{\rm{ms}}$ rather than estimating the sum of the errors $u - u^{\rm{ms}}$ and $u^{\rm{ms}} - u_h^{\rm{ms}}$. 

Assume the decomposition $u_h=u_h^B+u_h^\Har$ in its bubble and harmonic components, where $u_h^B\in V_h^B$, $u_h^\Har\in V_h^\Har$, and 
\begin{gather*}
V_h^B=\{v_h\in V_h:\,v_h=0\text{ on }\dtau,\,\tau\in\triH\},
\\
V_h^\Har=\{u_h^\Har\in V_h:\,a(u_h^\Har,v_h^B)=0\text{ for all }v_h^B\in V_h^B\}, 
\end{gather*}
i.e., $V_h^\Har=(V_h^B)^{\perp_a}$. It follows immediately from the definitions that
\begin{equation}\label{e:splitpdes}
a(u_h^\Har,v_h^\Har)=(\rho g,v_h^\Har)\quad\text{for all }v_h^\Har\in V_h^\Har, 
\qquad
a(u_h^B,v_h^B)=(\rho g,v_h^B)\quad\text{for all }v_h^B\in V_h^B. 
\end{equation}
The problems for the bubble solution $u_h^B$ are local and uncoupled and are considered in Section~\ref{s:localproblems}.

We now proceed to approximate $u_h^\Har$, and start by noting that the functions in $V_h^\Har$ are uniquely determined by their traces on the boundary of elements in $\triH$. Let 
\[
\Lambda_h=\{v_h|_{\btriH}:v_h\in V_h^\Har\}\subset H^{1/2}(\btriH), 
\]
and the \emph{local} discrete-harmonic extension operator $T:\Lambda_h\to V_h^\Har$ such that, for $\mu_h\in\Lambda_h$, 
\begin{equation}\label{e:definitionT}
  (T\mu_h)|_{\btriH}=\mu_h, 
  \qquad \text{and} \qquad
  a(T\mu_h,v_h^B)=0\quad\text{for all }v_h^B\in V_h^B. 
\end{equation}
Define the bilinear forms $s_\tau$, $s:\Lambda_h\times\Lambda_h\to\RR$ such that, for $\mu_h$, $\nu_h\in\Lambda_h$, 
\[
s_\tau(\mu_h,\nu_h)
=\int_\tau\A\bgrad T\mu_h\cdot\bgrad T\nu_h\,d\xx\quad\text{ for }\tau\in\triH,
\qquad
s(\mu_h,\nu_h)=\sum_{\tau\in\triH}s_\tau(\mu_h,\nu_h). 
\]
Let $\lambda_h=u_h|_{\btriH}$. Then $u_h^\Har=T\lambda_h$ and
\begin{equation}\label{e:pde}
s(\lambda_h,\mu_h)=(\rho g,T\mu_h)\quad\text{for all }\mu_h\in\Lambda_h. 
\end{equation}

\section{The Low-Contrast Multiscale Case} \label{s:lowcase}
We now propose a scheme to approximate~\eqref{e:pde} based on
LOD techniques. Define the fine-scale subspace $\tLambda\subset\Lambda_h$ by 
\begin{equation*}
\tLambda=\{\tlambda\in\Lambda_h:\,\tlambda(\xx_i)=0\text{ for all }\xx_i\in\NNN_H\}. 
\end{equation*}
Let the multiscale space $\Lambdams\subset\Lambda_h$ be such that $\tLambda\perp_s\Lambdams$ and $\Lambda_h=\tLambda\oplus\Lambdams$. Our numerical method is defined by $\lambda_h^{\rm{ms}}\in\Lambdams$ such that
\begin{equation}\label{e:lambdamsdef}
s(\lambda_h^{\rm{ms}},\mu_h^{\rm{ms}})=(\rho g,T\mu_h^{\rm{ms}})\quad\text{for all }\mu_h^{\rm{ms}}\in\Lambdams, 
\end{equation}
and we set $u_h^{\rm{ms}}=T\lambda_h^{\rm{ms}}$ as an approximation for $u_h^\Har$. 

To make the definition of $\Lambdams$ explicit, let the coarse-scale
space $\Lambda_H \subset \Lambda_h$ be the trace of piecewise continuous linear functions on the $\btriH$ triangulation. Thus, a function $\lambda_H\in\Lambda_H$ is uniquely determined by its nodal values and is linear on each edge. A basis $\{\theta_H^i\}_{i=1}^{\#\NNN_H}$ for $\Lambda_H$ can be obtained by imposing that  $\theta_H^i$ be continuous and piecewise linear on $\btriH$ and $\theta_H^i(\xx_j)=\delta_{ij}$ for all $\xx_j\in\NNN_H$. The support of $\theta_H^i$ is on all edges of elements $\tau\in\triH$ for which $\xx_i\in\bar\tau$. If $\mu_H=\sum_{i=1}^{\#\NNN_H}\mu_H(\xx_i) \theta_H^i$ is such that $\mu_H(\xx_i) = 0$ for all $\xx_i\in\NNN_H$, then $\mu_H(\xx)=0$ for all $\xx \in \NNN_h$. Hence, $\Lambda_h=\Lambda_H\oplus\tLambda$, and then $\dim\Lambda_H=\dim\Lambdams$.

Now, for each $K\in\triH$ and $\nu_h\in \Lambda_h$, let $P^K:\Lambda_h\to\tLambda$ be such that 
\begin{equation}\label{e:Ptaudef}
s(P^K\nu_h,\tmu)=s_K(\nu_h,\tmu)\quad\text{for all }\tmu\in\tLambda, 
\end{equation}
and $P:\Lambda_h\to\tLambda$ be such that
\begin{equation}\label{e:Pdef}
P\nu_h=\sum_{K\in\triH}P^K\nu_h.
\end{equation}
Note that
\[
s(P\nu_h,\tmu)
=\sum_{K\in\triH}s(P^K\nu_h,\tmu)
=\sum_{K\in\triH}s_K(\nu_h,\tmu)
=s(\nu_h,\tmu). 
\]
It follows from the above that $\Lambdams=\{(I-P)\theta_H:\,\theta_H\in\Lambda_H\}$.
A basis for $\Lambdams$ is defined by
$\lambda_i^{\rm{ms}}=(I-P)\theta_H^i \in\Lambdams$, and  by construction,
$\lambda_i^{\rm{ms}}(\xx_j)=\delta_{ij}$ for all $\xx_j\in\NNN_H$. 

An alternative to~\eqref{e:lambdamsdef} is to find $\lambda_H\in\Lambda_H$ such that 
\begin{equation}\label{e:lambdazerodef}
s\bigl((I-P)\lambda_H,(I-P)\mu_H\bigr)=(\rho g,T(I-P)\mu_H)\quad\text{for all }\mu_H\in\Lambda_H, 
\end{equation}
and then $\lambda_h^{\rm{ms}}=(I-P)\lambda_H$. We name it as ACMS--NLOD (\emph{Approximated Component Mode Synthesis Non-Localized Orthogonal Decomposition }) method. 

Albeit being well-defined, the method~\eqref{e:lambdazerodef} is not ``practical'', in the sense that the operators $P^K$ and  $P$ are nonlocal, and computing~\eqref{e:Ptaudef} is as hard as solving~\eqref{e:elliptic}. To circumvent that, we use the fact that the solutions of~\eqref{e:Ptaudef} actually decay exponentially to zero away from $K$. That allows the definition of a local approximation $P^{K,j}$ for $P^K$, having support at a patch of width $j$ around $K$. Next, before proving the exponential decay, we investigate the convergence rates for the ideal nonlocal solution $u_h^{\rm{ms}}$. 

In what follows, $\gamma_1$, $\gamma_2$, etc denote positive constants that do not depend on $\A$, $f$, $\rho$, $h$ and $H$, depending only on the shape regularity of elements on $\trih$ and $\triH$. Let 
\[
\quad\kappa=\max_{\tau\in\triH}\kappa^\tau,
\quad \kappa^\tau=\frac{\amax^\tau}{\amin^\tau},
\quad \amax^\tau=\sup_{\xx\in\tau}a_+(\xx),
\quad\amin^\tau=\inf_{\xx\in\tau}a_-(\xx),\]
\[
 \rhomax^\tau=\sup_{\xx\in\tau}\rho(\xx)
\quad \text{and} \quad \rhomin^\tau=\inf_{\xx\in\tau}\rho(\xx).
\]
Let us introduce the global Poincar\'e's inequality constant $C_{P,G}$
which is the smallest constant such that for all $\mu_h \in \Lambda_h$
\begin{equation}\label{e:umsest}
\|T\mu_h\|_{L_\rho^2(\Omega)}\le C_{P,G}|T\mu_h|_{H_\A^1(\Omega)}. 
\end{equation}
Let us also introduce the local Poincar\'e's inequality constant $c_{P,L}=\max_{\tau\in\triH}c_{P,L}^\tau$, where the $c_{P,L}^\tau$ are the smallest
constants such that 
\begin{equation}\label{e:lpi}
\|T\tmu\|_{L_\rho^2(\tau)}
\le c_{P,L}^\tau H|T\tmu|_{H_\A^1(\tau)}\quad\text{for all }\tmu\in\tLambda. 
\end{equation}
\begin{lemma}\label{l:WPI2}
Let $\tau\in\triH$ and $c_{P,L}^\tau$ as in~\eqref{e:lpi}. Then, an upper bound for $c_{P,L}^\tau$ is given by
\begin{equation} \label{e:poincare}
  (c_{P,L}^\tau)^2\le\gamma_1\left(1 +\log(H/h)\right)\frac{\rhomax^\tau}{\amin^\tau}.
\end{equation}
\end{lemma}
\begin{proof}
Using that $\tmu$ vanishes at the $\NNN_H$ nodes, we have~\cite{MR2104179} 
\begin{multline*}
\|T\tmu\|_{L_\rho^2(\tau)}^2
\le\gamma_1H^2\rhomax^\tau\|T\tmu\|_{L^\infty(\tau)}^2
\\
\le\gamma_1H^2\left(1 +\log(H/h)\right)\rhomax^\tau|T\tmu|_{H^1(\tau)}^2
\le\gamma_1H^2\left(1 +\log(H/h)\right)\frac{\rhomax^\tau}{\amin^\tau}|T\tmu|_{H_\A^1(\tau)}^2.
\end{multline*}
\end{proof}
\begin{lemma}\label{l:WPI}
Given $\mu_h\in\Lambda_h$ let $I_H\mu_h\in\Lambda_H$ be its Lagrange $\NNN_H$-nodal linear interpolation on $\btriH$. Then 
\begin{equation}\label{e:intenergstab}
|T I_H\mu_h|_{H_\A^1(\Omega)}^2
\le\gamma_2\kappa\bigl(1 +\log(H/h)\bigr)|T\mu_h|_{H_\A^1(\Omega)}^2. 
\end{equation}
\end{lemma}
\begin{proof}
  Let $T_\I$ be defined by~\eqref{e:definitionT} with $\A = \I$, the identity matrix. We first note that $|T_\I\mu_h|_{H^1(\tau)}\le|T \mu_h|_{H^1(\tau)}$, since $T_\I\mu_h$ minimizes $|\cdot|_{H^1(\tau)}$ over all functions in $H^1(\tau)$ with trace equal to $\mu_h$ over the element boundary. It follows~\cite{MR2104179} for each $\tau\in\triH$ that 
\begin{multline*} 
|T I_H\mu_h|_{H_\A^1(\tau)}^2
\le|T_\I I_H\mu_h|_{H_\A^1(\tau)}^2
\le\amax^\tau |T_\I I_H\mu_h|_{H^1(\tau)}^2
\le\gamma_2\amax^\tau\bigl(1 +\log(H/h)\bigr)|T_\I\mu_h|_{H^1(\tau)}^2
\\
\le\gamma_2\amax^\tau\bigl(1 +\log(H/h)\bigr)|T \mu_h|_{H^1(\tau)}^2
\le\gamma_2\kappa^\tau\bigl(1 +\log(H/h)\bigr)|T\mu_h|_{H_\A^1(\tau)}^2.  
\end{multline*}
\end{proof} 

We know extend the \emph{Face Lemma}~\cite[Subsection 4.6.3]{MR2104179} to variable coefficients.
\begin{lemma}\label{l:half_00rho}
Let $\tau\in\triH$, $e$ an edge of $\partial \tau$ and $\chi_e$ be the characteristic function of $e$ being identically equal to one on $e$ and zero on $\partial \tau \backslash e$. Then given $\tmu\in\tLambda$ we have
 \[
|T\chi_e\tmu|_{H_\A^1(\tau)}^2
\le\gamma_3\kappa^\tau\bigl(1 +\log(H/h)^2\bigr)|T\tmu|_{H_\A^1(\tau)}^2.
\]
\end{lemma}
\begin{proof}
We have
\begin{multline*} 
|T \chi_e\tmu|_{H_\A^1(\tau)}^2
\le|T_\I \chi_e\tmu|_{H_\A^1(\tau)}^2
\le\amax^\tau|T_\I\chi_e\tmu|_{H^1(\tau)}^2
\le\gamma_3\amax^\tau\bigl(1 +\log(H/h)\bigr)^2|T_\I\tmu|_{H^1(\tau)}^2
\\
\le \gamma_3\amax^\tau\bigl(1 +\log(H/h)\bigr)^2|T\tmu|_{H^1(\tau)}^2
\le \gamma_3\kappa^\tau\bigl(1 +\log(H/h)\bigr)^2|T\tmu|_{H_\A^1(\tau)}^2.  
\end{multline*}
\end{proof}

  %
%\begin{lemma}\label{l:WPI2}
%  Let $\tlambda\in\tLambda$. Then
%\begin{equation}\label{e:poincare}
%\|T\tlambda\|^2_{L_\rho^2(\Omega)}\le
%c(\rhomax/\amin) H^2 \log H/h |T\tlambda|^2_{H_\A^1(\Omega)}.
%\end{equation}
%\end{lemma}
%\begin{proof}
%refs? details?
%\end{proof}

\begin{theorem}\label{t:msgal}
Let $\lambda_h=u_h|_\btriH$, and $\lambda_h^{\rm{ms}}$ solution of~\eqref{e:lambdamsdef}. Then 
$\lambda_h-\lambda_h^{\rm{ms}}\in\tLambda$ and
\[
|u_h^\Har-u_h^{\rm{ms}}|_{H_\A^1(\Omega)}\le c_{P,L}H\|g\|_{L^2_{\rho}(\Omega)}, 
\]
where we recall that $u_h^\Har=T\lambda_h$ and $u_h^{\rm{ms}}=T\lambda_h^{\rm{ms}}$.
\end{theorem}
\begin{proof}
First note that $\lambda_h-\lambda_h^{\rm{ms}}\in\tLambda$ since it follows from the Galerkin orthogonality that $s(\lambda_h-\lambda_h^{\rm{ms}},\mu_h^{\rm{ms}})=0$ for all $\mu_h^{\rm{ms}}\in\Lambdams$. Using the local Poincar\'e's inequality~\eqref{e:lpi} we obtain  
\begin{multline*}
|u_h^\Har-u_h^{\rm{ms}}|_{H_\A^1(\Omega)}^2
=s(\lambda_h-\lambda_h^{\rm{ms}},\lambda_h-\lambda_h^{\rm{ms}})
=s(\lambda_h-\lambda_h^{\rm{ms}},\lambda_h)
=\bigl(\rho g,T(\lambda_h-\lambda^{\rm{ms}})\bigr)
\\
\le\|g\|_{L_\rho^2(\Omega)}\|T(\lambda_h-\lambda_h^{\rm{ms}})\|_{L^2_\rho(\Omega)}
\le c_{P,L}H\|g\|_{L_\rho^2(\Omega)}|T(\lambda_h-\lambda_h^{\rm{ms}})|_{H_\A^1(\Omega)}, 
\end{multline*}
and the result follows.
\end{proof}

\subsection{Decaying Results for Low-Contrast coefficients} \label{ss:lowcon}
We next prove exponential decay of $P^K\nu_h$ for $K\in\triH$. Denote
\[
\T_1(K)=\{K\}, 
\qquad
\T_{j+1}(K)
=\{\tau\in\triH:\,\overline\tau\cap\overline\tau_j\ne\emptyset\text{ for some }\tau_j\in\T_j(K)\}. %\label{p:submeshdef}
\]
The following estimate is fundamental to prove exponential decay.

\begin{lemma}\label{l:decay}
Assume that $K\in\triH$ and $\nu_h\in\Lambda_h$, and let $\tilde\phi_h=P^K\nu_h\in\tLambda$. Then, for any integer $j\ge 1$, 
\begin{equation*}
|T\tilde{\phi}_h|_{H_\A^1(\triH\backslash\T_{j+1}(K))}^2
\le 9 \alpha |T\tilde{\phi}_h|_{H_\A^1(\T_{j+1}(K)\backslash\T_j(K))}^2, 
\end{equation*}
where $\alpha=\gamma_3\kappa (1+ \log(H/h))^2$. 
\end{lemma}
\begin{proof} 
Choose $\tnu\in\tLambda$ such that $\tnu|_{\dtau}=\tilde{\phi}_h$ if $\tau\in\triH\backslash\T_{j+1}(K)$, and $\tnu=0$ on the remaining edges. We obtain 
\begin{multline*}
|T\tilde{\phi}_h|_{H_\A^1(\triH\backslash\T_{j+1}(K))}^2
=s_K(\tnu,\nu_h)
-\sum_{\tau\in\T_{j+1}(K)\backslash\T_j(K)}s_\tau(\tnu,\tilde{\phi}_h)
=-\sum_{\tau\in\T_{j+1}(K)\backslash\T_j(K)}s_\tau(\tnu,\tilde{\phi}_h)
\\
\le\sum_{\tau\in\T_{j+1}(K)\backslash\T_j(K)}|T\tnu|_{H_\A^1(\tau)}|T\tilde{\phi}_h|_{H_\A^1(\tau)}, 
\end{multline*}
where we used that  $s(\tnu,\tilde{\phi}_h)=s_K(\tnu,\nu_h)=0$ since 
the support of $\tnu$ does not intersect with $K$. For each edge $e$ of $\dtau$, let $\chi_e$ be the
characteristic function of $e$ being identically equal to one on $e$ and
zero on $\partial \tau \backslash e$. For $\tau\in\T_{j+1}(K)\backslash\T_j(K)$, 
\[
|T\tnu|_{H_\A^1(\tau)}^2 
\le3\sum_{e\subset\dtau}|T(\chi_e\tnu)|_{H_\A^1(\tau)}^2
\le9\gamma_3\kappa^\tau (1+ \log H/h)^{2}|T\tilde{\phi}_h|_{H_\A^1(\tau)}^2,
\]
where we have used the \emph{Face Lemma}~\cite[Subsection 4.6.3]{MR2104179}.
\end{proof}
\begin{corollary}\label{c:decay}
 Assume that $K\in\triH$ and $\nu_h\in\Lambda_h$ and let
 $\tilde{\phi}_h=P^K \nu_h\in\tLambda$. Then, for any integer $j\ge 1$, 
\begin{equation*}
  |T\tilde{\phi}_h|_{H_\A^1(\triH\backslash\T_{j+1}(K))}^2
  \le e^{-\frac{j}{1+9\alpha}} |T\tilde{\phi}_h|_{H_\A^1(\triH)}^2, 
\end{equation*}
where $\alpha$ is as in Lemma~\ref{l:decay}. 
\end{corollary}
\begin{proof} 
Using Lemma~\ref{l:decay} we have 
\begin{equation*}
|T\tilde{\phi}_h|_{H_\A^1(\triH\backslash{\T}_{j+1}(K))}^2  \le 9\alpha|T\tilde{\phi}_h|_{H_\A^1(\triH\backslash{\T}_{j}(K))}^2-9\alpha|T\tilde{\phi}_h|_{H_\A^1(\triH\backslash{\T}_{j+1}(K))}^2
\end{equation*}
and then
\[
|T\tilde{\phi}_h|_{H_\A^1(\triH\backslash{\T}_{j+1}(K))}^2
\le\frac{9\alpha}{1+9\alpha}|T\tilde{\phi}_h|_{H_\A^1(\triH\backslash{\T}_{j}(K))}^2 \le e^{-\frac{1}{1+9\alpha}} |T\tilde{\phi}_h|_{H_\A^1(\triH\backslash{\T}_{j}(K))}^2,
\]
and the theorem follows.
\end{proof} 

\begin{rem} \label{r:rem1}
  The $\alpha$ in this paper, defined in Lemma~\ref{l:decay}, is estimated as the worst case scenario. For
  particular cases of coefficients $\A$ and $\rho$,  sharper estimates 
  for $\alpha$ can  be derived using weighted
  Poincar\'e inequalities techniques and  partitions of unity
  that conform with $\A$ in order to avoid
  large energies on the interior extensions~\cite{MR1367653,MR2867661,MR3225627,MR3013465,MR3047947,MR2456834,MR2810804,MR2861254}; see~\cite{hellman17_mms,MR3552482} for examples.
\end{rem}

Inspired by the exponential decay stated in Corollary~\ref{c:decay}, we define the operator $P^j$ as follows. First, for a fixed $K\in\triH$, let
\[
\tLambda^{K,j}=\{\tmu\in\tLambda:\,T\tmu=0\text{ on }\triH\backslash\T_j(K)\}. 
\]
Given $\mu_h \in \Lambda_h$, define then $P^{K,j}\mu_h\in\tLambda^{K,j}$ such that 
\[
s(P^{K,j}\mu_h,\tmu)=s_K(\mu_h,\tmu)\quad\text{for all }\tmu\in\tLambda^{K,j}, 
\]
and let
\begin{equation}\label{e:Pjdef}
P^j\mu_h=\sum_{K\in\triH}P^{K,j}\mu_h. 
\end{equation}

We define the approximation $\lambda_H^j\in\Lambda_H$ of $\lambda_H$ by
\begin{equation}\label{e:lambdamsjdef}
s\bigl((I-P^j)\lambda_H^{j},(I-P^j)\mu_H\bigr)
=(\rho g,T(I-P^j)\mu_H)\quad\text{for all }\mu_H\in\Lambda_H, 
\end{equation}
and then let $\lambda_h^{\rm{ms},j}=(I-P^j)\lambda_H^{j}$ and $u^{\rm{ms},j}_h=T\lambda_h^{\rm{ms},j}$. We name the scheme as ACMS--LOD (\emph{Approximated Component Mode Synthesis Localized Orthogonal Decomposition}) method. 

We now analyze the approximation error of the method, starting by a technical result essential to obtain the final estimate. Let $c_\gamma$ be a constant depending only on the shape regularity of $\triH$ such that
\begin{equation}\label{e:cgammadef}
\sum_{\tau\in\triH}|v|_{H^1(\T_{j}(\tau))}^2\le(c_\gamma j)^2|v|_{H^1(\triH)}^2, 
\end{equation}
for all $v\in H^1(\triH)$.

\begin{lemma}\label{l:errorPj}
Consider $\nu_h\in\Lambda_h$ and the operators $P$ defined by~\eqref{e:Pdef} and $P^j$ by~\eqref{e:Pjdef} for $j>1$. Then
\begin{equation*}
|T(P-P^j)\nu_h|_{H_\A^1(\triH)}^2
\le (9c_\gamma j\alpha)^2e^{-\frac{j-2}{1+9\alpha}}|T\nu_h|^2_{H_\A^1(\triH)}.
\end{equation*}
\end{lemma}
\begin{proof}
Let $\tilde\psi_h=(P-P^j)\nu_h=\sum_{K\in\triH}(P^K-P^{K,j})\nu_h$. For each $K\in\triH$, let $\tilde\psi_h^K\in\tilde\Lambda_h$ be such that $\tilde\psi_h^K|_e=0$ if $e$ is a face of an element of $\T_j(K)$ and $\tilde\psi_h^K|_e=\tilde\psi_h|_e$, otherwise. We obtain 
\begin{equation}\label{e:splitting}
|T\tilde\psi_h|_{H_\A^1(\triH)}^2
=\sum_{K\in\triH}\sum_{\tau\in\triH} 
s_\tau(\tilde\psi_h-\tilde\psi_h^K,(P^K-P^{K,j})\nu_h)
+s_\tau(\tilde\psi_h^K,(P^K-P^{K,j})\nu_h). 
\end{equation}
See that the second term of~\eqref{e:splitting} vanishes since 
\[
\sum_{\tau\in\triH}s_\tau(\tilde\psi_h^K,(P^K-P^{K,j})\nu_h)_\dtau
=\sum_{\tau\in\triH}s_\tau(\tilde\psi_h^K,P^K\nu_h)_\dtau=0.
\]
For the first term of~\eqref{e:splitting},  
\begin{multline*}
\sum_{\tau\in\triH}s_\tau(\tilde\psi_h-\tilde\psi_h^K,(P^K-P^{K,j})\nu_h)_{\partial\tau}
\le\sum_{\tau\in\T_{j+1}(K)}
|T(\tilde\psi_h-\tilde\psi_h^K)|_{H_\A^1(\tau)}
|T(P^K-P^{K,j})\nu_h|_{H_\A^1(\tau)}
\\
\le|T\tilde\psi_h-\tilde\psi_h^K|_{H_\A^1(\T_{j+1}(K))}
  |T(P^K-P^{K,j})\nu_h|_{H_\A^1(\T_{j+1}(K))}.
\end{multline*}
However, proceeding as in the proof of Lemma~\ref{l:decay}, we gather that
\begin{equation*}
|T\tilde\psi_h-\tilde\psi_h^K|_{H_\A^1(\T_{j+1}(K))}
\le 3\alpha^{1/2}|T\tilde\psi_h|_{H_\A^1(\T_{j+1}(K))}   
\end{equation*}
Let $\nu_h^{K,j}\in\tLambda^{K,j}$ be equal to zero on all faces of elements of $\T_H\backslash\T_j(K)$ and equal to $P^K\nu_h$ otherwise. Using Galerkin best approximation property and Corollary~\ref{c:decay} we obtain
\begin{multline*}
|T(P^K-P^{K,j})\nu_h|_{H_\A^1({\T}_{j+1}(K))}^2
\le|T(P^K-P^{K,j})\nu_h|_{H_\A^1(\triH)}^2
\le|T(P^K\nu_h-\nu_h^{K,j})|_{H_\A^1(\triH)}^2
\\ 
\le9\alpha|TP^K\nu_h|_{H_\A^1(\triH\backslash\T_{j-1}(K))}^2 
\le9\alpha e^{-\frac{j-2}{1+9\alpha}}|TP^K\nu_h|_{H_\A^1(\triH)}^2.
\end{multline*}
We gather the above results to obtain  
\begin{multline*}
|T\tilde\psi_h|_{H_\A^1(\triH)}^2
\le9\alpha e^{-\frac{j-2}{2(1+9\alpha)}}
\sum_{K\in\triH}|T\tilde\psi_h|_{H_\A^1(\T_{j+1}(K))}|TP^K\nu_h|_{H_\A^1(\triH)}
\\
\le 9\alpha e^{-\frac{j-2}{2(1+9\alpha)}}c_\gamma j|T\tilde\psi_h|_{H_\A^1(\triH)}
\biggl(\sum_{K\in\triH}|TP^K\nu_h|_{H_\A^1(\triH)}^2\biggr)^{1/2}. 
\end{multline*}

We finally gather that 
\[
|TP^K\nu_h|_{H_\A^1(\triH)}^2
=s(P^K\nu_h,P^K\nu_h)_{\partial\triH}
=s_K(P^K\nu_h,\nu_h)
=\int_K\A\bgrad(TP^K\nu_h)\cdot\bgrad T\nu_h\,d\xx
\]
and from Cauchy--Schwarz, $|TP^K\nu_h|_{H_\A^1(\triH)}\le|T\nu_h|_{H_\A^1(K)}$,
we have
\[
\sum_{K\in\triH}|TP^K\nu_h|_{H_\A^1(\triH)}^2\le|T\nu_h|_{H_\A^1(\triH)}^2.
\]
\end{proof}

\begin{theorem}\label{t:lowconterror}
Define $u_h^\Har$ by~\eqref{e:splitpdes} and let $u^{\rm{ms},j}_h=T(I-P^j)\lambda_H^j$, where $\lambda_H^j$ is as in~\eqref{e:lambdamsjdef}. Then
\[
|u_h^\Har-u_h^{\rm{ms},j}|_{H_\A^1(\triH)}
\le H\bigl\{c_{P,L}+[\gamma_2\kappa\left(1 +\log(H/h)\right)]^{1/2} c_\gamma j9\alpha e^{-\left(\frac{j-2}{2(1+9\alpha)}-\log(c_{P,G}/H)\right)}\bigr\}\|g\|_{L_\rho^2(\Omega)}. 
\]
\end{theorem}
\begin{proof}
First, from the triangle inequality, 
\[
|u_h^\Har-u^{\rm{ms},j}_h|_{H_\A^1(\triH)}
\le|u_h^\Har-u^{\rm{ms}}_h|_{H_\A^1(\triH)}+|u_h^{\rm{ms}}-u^{\rm{ms},j}_h|_{H_\A^1(\triH)}, 
\]
and for the first term we use Theorem~\ref{t:msgal}. For the second term, we first define $\hat u^{\rm{ms},j}_h=\sum_{i}\lambda^{\rm{ms}}_h(\xx_i)T(I-P^j)\theta_H^i$, and then
\[
u^{\rm{ms}}_h-\hat u^{\rm{ms},j}_h
=T(P-P^j)\sum_{i}\lambda_h^{\rm{ms}}(\xx_i)\theta_H^i
=T(P-P^j)I_H\lambda^{\rm{ms}}_h, 
\]
where $I_H$ is as in Lemma~\ref{l:WPI}. Relying on the Galerkin best approximation we gather from Lemma~\ref{l:errorPj} that
\[
|u_h^{\rm{ms}}-u_h^{\rm{ms},j}|_{H_\A^1(\triH)}^2
\le|u_h^{\rm{ms}}-\hat{u}^{\rm{ms},j}_h|_{H_\A^1(\triH)}^2
\le (c_\gamma j)^2(9\alpha)^2e^{-\frac{j-2}{(1+9\alpha)}}|TI_H\lambda_h^{\rm{ms}}|_{H_\A^1(\triH)}^2.
\]
Since $u^{\rm{ms}}_h=T\lambda^{\rm{ms}}_h$ the result follow from Lemma~\ref{l:WPI} and the global Poincar\'e's inequality~\eqref{e:umsest}. 
\end{proof}

\begin{rem} Note that nothing precludes the use of quadrilateral meshes and other discretization schemes ($hp$ for instance), but some of the constants in the error estimates would change accordingly. 
\end{rem}

\section{The High-Contrast Multiscale Case} \label{s:highcase}
The main bottle-neck in dealing with high-contrast coefficients is that $\alpha$ becomes too large, therefore $j$ has to be large as well, cf. Theorems~\ref{t:msgal} and~\ref{t:lowconterror}. Furthermore, the large local Poincar\'e inequality constant $c_{P,L}^\tau$ deteriorates the a priori error estimate in Theorem~\ref{t:msgal}. Also, we would like to remove the $(1 + \log(H/h)$ term that
appears in these estimates due to the mismatch between $H^{1/2}(e)$ and $H_{00}^{1/2}(e)$ (see~\cite{MR2104179} for a definition). To deal with these issues, we replace $\widetilde{\Lambda}_h$ by a subspace  ${\Lambda}_h^\triangle\subset\tLambda$ by removing
a subspace spanned by some eigenfunctions associated to an appropriate generalized eigenvalue problem, on each edge of the mesh $\triH$. We first introduce some notation.

Given an edge $e$ of an element $\tau\in\triH$, let $\tLambda^e=\tLambda|_e$ and  $\tLambda^\tau=\tLambda|_\dtau$ be the restrictions of functions on $\tLambda$ to $e$ and on $\tLambda$ to $\dtau$. Since $\tmu^e\in\tLambda^e$ vanishes at the end-points of $e$, it is possible to continuously extend it by zero for all nodes $\xx_i\in \NNNEc:= (\NNN_h \backslash \NNN_H)\cap (\dtau \backslash e)$. Let $R_{e,\tau}^T:\tLambda^e\rightarrow\tLambda^\tau$ be such extension. Conversely, we define the restriction operator $R_{e,\tau}:\tLambda^\tau\to\tLambda^e$ such that $R_{e,\tau}\nu_h(\xx_i)= \nu_h(\xx_i)$ for all nodes $\xx_i\in \NNNE:= (\NNN_h\backslash\NNN_H)\cap e$.

Denote by $(\cdot,\cdot)_e$ the $L^2(e)$ inner product and define $S^\tau:\,\tLambda^\tau\to(\tLambda^\tau)'$, where $(\tLambda^\tau)'$ is the dual space of $\tLambda^\tau$, such that 
\[
(\mu^\tau_h, S^\tau \nu^\tau_h)_{\partial \tau}
=\int_\tau\A \bgrad T \mu_h^\tau\cdot\bgrad T\nu^\tau_h\,d\xx
\qquad\text{for all }\mu^\tau_h,\nu^\tau_h\in\tLambda^\tau.
\]
Also let $S_{ee}^\tau:\tLambda^e\to(\tLambda^e)'$ be such that
\[
(\tmu^e,S^\tau_{ee}\tnu^e)_e=
(R_{e,\tau}^T \tmu^e,S^\tau  R_{e,\tau}^T
\tnu^e)_{\partial\tau}
\quad\text{for all }\tmu^e,\tnu^e\in\tLambda^e,
\]
Similarly we define $S_{e^ce}^\tau$, $S_{ee^c}^\tau$ and $S_{e^ce^c}^\tau$, related to the degrees of freedom on $e^c=\NNNEc$. 

Let us introduce $M_{ee}^\tau$ by 
\[
(\tilde{\mu}_h^e, M^\tau_{ee} \tilde{\nu}_h^e)_e =
 \int_\tau \rho \,(T R_{e,\tau}^T
\tilde{\mu}_h^e)
\, (T
R_{e,\tau}^T \tilde{\nu}_h^e)\,d\xx
\]
and define $\widehat{S}_{ee}^\tau=\HH^{-2}\,M_{ee}^\tau+S_{ee}^\tau$, where
$\HH$ is the target precision of the method, that can be set by the user.

We finally consider the Schur complement 
\[
\widetilde S_{ee}^\tau=S_{ee}^\tau-S_{ee^c}^\tau(S_{e^ce^c}^\tau)^{-1}S_{e^ce}^\tau, 
\]
and then 
\begin{equation}\label{e:tildeSest}
(\tnu^e,\widetilde S_{ee}^\tau\tnu^e)
\le(\nu_h,S^\tau\nu_h)
\quad\text{for all }\nu_h\in\tLambda^\tau\text{ such that }R_{e,\tau}\nu_h=\tnu^e. 
\end{equation}
See~\cite{MR4238010} for a similar computation.

We are ready then to define a generalized eigenvalue problem that takes into account high contrast coefficients. For a given edge $e$ shared by elements $\tau$ and $\tau^\prime$, find eigenpairs $(\alpha_i^e,\tilde\psi_{h,i}^e)\in(\RR,\widetilde\Lambda_h^e)$, where
$\alpha_1^e\ge\alpha_2^e\ge\alpha_3^e\ge\dots\ge\alpha_{\NNNE}^e>1$, such that 
%\[
%\widehat{S}^\tau_{ee}:\widehat{S}^{\tau^\prime}_{ee}  \tilde{\psi}_{h,i}^e
%= \lambda \widetilde{S}^\tau_{ee}:\widetilde{S}^{\tau^\prime}_{ee} \tilde{\psi}%_{h,i}^e.
%\]
%where we have denoted $A:B = (A^{-1} + B^{-1})^{-1}$
\begin{equation}\label{e:eigenvalue}
(\widehat{S}_{ee}^\tau+\widehat{S}_{ee}^{\tau^\prime}) \tilde{\psi}_{h,i}^e
= \alpha_i^e (\widetilde{S}^\tau_{ee} +\widetilde{S}^{\tau^\prime}_{ee}) \tilde{\psi}_{h,i}^e.
\end{equation}
We impose that the eigenfunctions $\tilde{\mu}_{h,i}^e$ are orthonormal with respect to $(\cdot, (\widehat{S}_{ee}^\tau + \widehat{S}_{ee}^{\tau^\prime})\cdot)_e$. To see that all eingenvalues are greater than one, note from~\eqref{e:tildeSest} and the above definitions that for all $\tilde{\psi}_h^e\in\widetilde\Lambda_h^e$, 
\[
(\tilde{\psi}_h^e, \widetilde{S}^\tau_{ee}\tilde{\psi}_h^e)_e
\le(R_{e,\tau}^T\tilde{\psi}_h^e,S^\tau R_{e,\tau}^T\tilde{\psi}_h^e)
=(\tilde{\psi}_h^e,S_{ee}^\tau\tilde{\psi}_h^e)
<(\tilde{\psi}_h^e,\widehat{S}_{ee}^\tau\tilde{\psi}_h^e), 
\]
where we recall that $R_{e,\tau}^T\tilde{\psi}_h^e$ is the extension of $\tilde{\psi}_h^e$ by zero. 

Now we decompose $\tLambda^e:=\tLambda^{e,\triangle}\oplus\tLambda^{e,\Pi}$ where for a given $\alpha_{\rm{stab}}>1$, 
\begin{equation} \label{e:Lpmdef}
\widetilde\Lambda_h^{e,\triangle}:=\span\{\tilde\mu_{h,i}^e:\,\alpha_{i}^e < \alpha_{\rm{stab}}\},
\qquad
\widetilde\Lambda_h^{e,\Pi}:=\span\{\tilde\mu_{h,i}^e:\,\alpha_{i}^e \ge \alpha_{\rm{stab}}\}.
\end{equation}
We remark that ${\alpha}_{\rm{stab}}$ is chosen by the user and replaces ${\alpha}$ in the proof of Lemma~\ref{l:decay2}, the counterpart of Lemma~\ref{l:decay}.

To define our ACMS--NLSD (\emph{Approximated Component Mode Synthesis Non-Localized Spectral Decomposition }) method for high-contrast coefficients, let 
\begin{equation}\label{e:ltpmdef}
\begin{gathered}
\tLambda^\Pi
=\{\tilde\mu_h\in\widetilde\Lambda_h:\,\tilde\mu_h|_e\in\tLambda^{e,\Pi}\text{ for all }e\in\btriH\},
\\ 
\tLambda^\triangle
=\{\tilde\mu_h\in\widetilde\Lambda_h:\,\tilde\mu_h|_e\in\tLambda^{e,\triangle}\text{ for all }e\in\btriH\}. 
\end{gathered}
\end{equation}

Note that $\Lambda_h=\Lambda_h^\Pi\oplus\tLambda^\triangle$, where 
\[
\Lambda_h^\Pi =  \Lambda_h^0 \oplus \widetilde\Lambda_h^\Pi
\]
and $\Lambda_h^0$ is the set of functions on $\Lambda_h$ which vanish on all
nodes of $\NNN_h \backslash \NNN_H$.
Denote
\[
({\nu}_h,S \mu_h)_{\btriH} = \sum_{\tau \in \triH}
({\nu}_h^\tau,S^\tau \mu^\tau_h)_{\partial \tau}.
\]

We now introduce the ACMS--NLSD multiscale functions. For $\tau\in\triH$, consider the operators $P^{\tau,\triangle}$, $P^\triangle:\Lambda_h\rightarrow\tLambda^\triangle$ as follows: Given
$\mu_h\in\Lambda_h$, find $P^{\tau,\triangle}\mu_h\in\tLambda^\triangle$ and define $P^\triangle$ such that 
\begin{equation} \label{e:Pdelta} 
(\tnu^\triangle,SP^{\tau,\triangle}\mu_h)_\btriH
=(\tnu^\triangle,S^\tau\mu_h)_\dtau
\quad\text{for all }\tnu^\triangle\in\tLambda^\triangle,
\qquad
P^\triangle=\sum_{\tau\in\triH}P^{\tau,\triangle}. 
\end{equation}

Consider $\Lambda_h^{\rm{ms},\Pi} = (I - P^\triangle)\Lambda_h^\Pi$. The ACMS--NLSD method is defined by: Find $\lambda_h^{\rm{ms},\Pi}\in\Lambda_h^{\Pi,\rm{ms}}$ such that
\begin{equation}\label{e:acms-lod}
(\nu_h^{\rm{ms},\Pi},S\lambda_h^{\rm{ms},\Pi})_{\btriH}=(\rho g,T\nu_h^{\rm{ms},\Pi})
\quad\text{for all }\nu_h^{\rm{ms},\Pi}\in\Lambda_h^{\rm{ms},\Pi}.
\end{equation}
Note that 
\[
(\nu_h^{\rm{ms},\Pi}, S \lambda_h^{\rm{ms},\Pi})_{\btriH}
=\int_\Omega\A\bgrad T\nu_h^{\rm{ms},\Pi}\cdot\bgrad T\lambda_h^{\rm{ms},\Pi}\,d\xx
=\int_\Omega\rho gT\nu_h^{\rm{ms},\Pi}\,d\xx. 
\]

\begin{rem}
A similar approach was followed by~\cite{MR2666649,MR3225627}, where different local eigenvalue problems are introduced to construct the approximation spaces. The analysis of the method however requires extra regularity of the coefficients, and the error estimate is not robust with respect to contrast. 
\end{rem}

\begin{rem} \label{r:rho} 
  Assume that $\A$ is scalar and $\rho = \A \ge 1$. The generalized eigenvalue
  problem (\ref{e:eigenvalue}) was designed to guarantee that both the 
  exponential decay of the multiscale function basis and local weighted
  Poincar\'{e} inequality holds,
  \[
  \|v\|_{L^2_\rho (\tau \cup \tau^\prime)} \leq c \HH \|v\|_{H^1_\A(\tau \cup \tau^\prime)}, 
  \]
  without hidden constants depending of $\A$ or on the contrast of $\A$.
  If the issue were just the
  exponential decay, it is possible to show from the analysis that we could
  have defined $\hat{S}_{ee}^\tau = S_{ee}^\tau$, furthermore, to show that the
  number of large eigenvalues is related to the number of high permeable
  channels crossing the edge $e$; see \cite{MR4683905, MR4320895}.
  These eigenvectors are similar to
  functions that have value equal to one in one channel, zero value in the
  other channels, and smooth on the low permeable region. Related argument
  also applies to define subspaces where the weighted Poincar\'e inequality
  holds. In summary, we can control both the exponential decay and the
  local weighted Poincar\'e by using similar number of eigenvectors if using
  the classical Poincar\'e inequality. As a consequence, the error estimate
  in terms of $\|f\|_{L^2_{1/\rho}}$ is a stronger result than in terms of $\|f\|_{L^2}$ and using similar number of multiscale function basis. 
   \end{rem} 

The counterpart of Lemma~\ref{l:WPI2} follows. 
\begin{lemma}\label{l:WPI3}
Let $\tmu^\triangle\in\tLambda^\triangle$. Then
\begin{equation}\label{e:poincaredelta}
\|T\tmu^\triangle\|_{L_\rho^2(\Omega)}
\le(9\alpha_{\rm{stab}})^{1/2}\HH|T\tmu^\triangle|_{H_\A^1(\Omega)}.
\end{equation}
\end{lemma}
\begin{proof}
We have for $\tau\in\triH$, 
\[
\HH^{-2}\|T\tmu^\triangle\|_{L_\rho^2(\tau)}^2
\le3\HH^{-2}\sum_{e\subset\dtau}\|TR_{e,\tau}^T\tmu^{e,\triangle}\|_{L_\rho^2(\tau)}^2. 
\]
Fixing  the edge $e$ of both $\tau$ and $\tau^\prime$, we have 
\begin{multline*} 
\HH^{-2}\|TR_{e,\tau}^T\tmu^{e,\triangle}\|^2_{L_\rho^2(\tau)}+\HH^{-2}\|TR_{e,\tau^\prime}^T\tmu^{e,\triangle}\|^2_{L_\rho^2(\tau^\prime)}
\le(\tmu^{e,\triangle},\widehat S_{ee}^\tau\tmu^{e,\triangle})_e+(\tmu^{e,\triangle},\widehat S_{ee}^{\tau^\prime}\tmu^{e,\triangle})_e
\\
\le\alpha_{\rm{stab}}\bigl(
\tilde{\mu}_h^{e,\triangle},(\widetilde{S}_{ee}^\tau
+\widetilde{S}^{\tau^\prime}_{ee})\tmu^{e,\triangle}\bigr)_e
\le\alpha_{\rm{stab}}
\bigl(|T\tmu^{\tau,\triangle}|_{H^1_\A(\tau)}^2+|T\tmu^{\tau^\prime,\triangle}|_{H^1_\A(\tau^\prime)}^2 \bigr)
\end{multline*}
from~\eqref{e:eigenvalue},~\eqref{e:Lpmdef} and~\eqref{e:tildeSest}. 
 By adding all $\tau \in \triH$, the results follows.
\end{proof}
Note that we added $\HH^{-2} M_{ee}^\tau$ to define $\widehat{S}^\tau_{ee}$. This
is necessary otherwise we might have a few modes that would make 
the local Poincar\'e's inequality constant in \eqref{e:poincaredelta}
too large. 

Now we concentrate on the counterpart of Lemma~\ref{l:WPI}.
\begin{lemma}\label{l:WPIhigh}
Let $\mu_h\in\Lambda_h$ and let $\mu_h=\mu_h^\Pi+\tmu^\triangle$. Then
\[
|T\mu_h^\Pi|_{H^1_\A(\Omega)} \le (2 + 18 \alpha_{\rm{stab}})^{1/2} |T\mu_h|_{H^1_\A(\Omega)}.
\]
\end{lemma}
\begin{proof}
We have
\[
|T\mu_h^\Pi|^2_{H^1_\A(\Omega)}
\le2(|T\mu_h|_{H^1_\A(\Omega)}^2+|T\tmu^\triangle|_{H^1_\A(\Omega)}^2).
\]
Consider the decomposition
\[
\mu_h|_\tau  = \mu_h^{\tau,0} + \sum_{e \subset \partial \tau}\tmu^{\tau, e}
\]
where $\mu_h^{\tau,0}\in\Lambda_h^0$ and $\tilde{\mu}_h^{\tau, e} =
\tilde{\mu}_h^{\tau, e, \Pi} + \tilde{\mu}_h^{\tau, e, \triangle}$. 
Then
\[
|T\tilde{\mu}_h^{\tau,\triangle}|_{H^1_\A(\tau)}^2
\le 3 \sum_{e \subset \partial \tau}|TR_{e,\tau}^T\tilde{\mu}_h^{\tau,e,\triangle}|_{H^1_\A(\tau)}^2
= 3\sum_{e\subset\dtau}(\tilde{\mu}_h^{\tau,e,\triangle},{S}_{ee}^\tau \tilde{\mu}_h^{\tau,e,\triangle})_e.
\]
Now we use that, if $e$ is an edge of both $\tau$ and $\tau^\prime$, 
\[
\left(\tilde{\mu}_h^{\tau,e,\triangle},  ({S}_{ee}^\tau + {S}_{ee}^{\tau^\prime}) \tilde{\mu}_h^{\tau,e,\triangle}\right)_e \leq
\alpha_{\rm{stab}}
\left(\tilde{\mu}_h^{\tau,e,\triangle},  (\widetilde{S}_{ee}^\tau + \widetilde{S}_{ee}^{\tau^\prime}) \tilde{\mu}_h^{\tau,e,\triangle}\right)_e.
\]
In addition, due to the orthogonality condition of the spaces $\widetilde{\Lambda}_h^{e,\triangle}$ and $\widetilde{\Lambda}_h^{e,\Pi}$ with respect to the inner product $(\cdot,(\widetilde{S}_{ee}^\tau+\widetilde{S}_{ee}^{\tau^\prime})\cdot)_e$,
and ~\eqref{e:tildeSest}, 
we have
\[
\left(\tilde{\mu}_h^{\tau,e,\triangle},  (\widetilde{S}_{ee}^\tau + \widetilde{S}_{ee}^{\tau^\prime}) \tilde{\mu}_h^{\tau,e,\triangle}\right)_e \leq
\left(\tilde{\mu}_h^{\tau,e},  (\widetilde{S}_{ee}^\tau + \widetilde{S}_{ee}^{\tau^\prime}) \tilde{\mu}_h^{\tau,e}\right)_e 
\le |T\mu_h^\tau|_{H^1_\A(\tau)}^2 +|T{\mu}_h^{\tau^\prime}|_{H^1_\A(\tau^\prime)}^2. 
\]
Adding all terms together we obtain the result.
\end{proof}

We now state the counterpart of the \emph{Face Lemma}~\cite[Subsection 4.6.3]{MR2104179}. The lemma follows directly from the definition of the generalized eigenvalue problem and properties of $\widetilde{\Lambda}_h^{\tau,e,\triangle}$ and~\eqref{e:tildeSest}.
\begin{lemma}\label{l:half_00}
Let $e$ be a common edge of $\tau$, $\tau^\prime\in\triH$, and $\tmu^\triangle\in\tLambda^\triangle$. Then, defining $\tmu^{e,\triangle}=R_{e,\tau}\tmu^\triangle$ and $\tmu^{\tau,\triangle}=\tmu^\triangle|_{\dtau}$ it follows that 
\[
|TR_{e,\tau}^T\tmu^{e,\triangle}|_{H_\A^1(\tau)}^2
+|TR_{e,\tau^\prime}^T\tmu^{e,\triangle}|_{H^1_\A(\tau^\prime)}^2
\le\alpha_{\rm{stab}}\bigl(
|T\tmu^{\tau,\triangle}|_{H^1_\A(\tau)}^2+|T\tmu^{\tau^\prime,\triangle}|_{H^1_\A(\tau^\prime)}^2 
\bigr)
\]
\end{lemma}
\begin{proof}
We have
\begin{multline*} 
|TR_{e,\tau}^T\tmu^{e,\triangle}|_{H_\A^1(\tau)}^2
+|TR_{e,\tau^\prime}^T\tmu^{e,\triangle}|_{H^1_\A(\tau^\prime)}^2
=(\tmu^{e,\triangle},S_{ee}^\tau\tmu^{e,\triangle})_\dtau+(\tmu^{e,\triangle},S_{ee}^{\tau^\prime}\tmu^{e,\triangle})_{\dtau^\prime}
\\
\le\alpha_{\rm{stab}}(\tmu^{e,\triangle},\widetilde S_{ee}^\tau\tmu^{e,\triangle})_\dtau+(\tmu^{e,\triangle},\widetilde S_{ee}^{\tau^\prime}\tmu^{e,\triangle})_{\dtau^\prime}
 \le\alpha_{\rm{stab}}(\tmu^\triangle,S^\tau\tmu^\triangle)_\dtau+(\tmu^\triangle,S^{\tau^\prime}\tmu^\triangle)_{\dtau^\prime}. 
\end{multline*}
\end{proof}

\begin{theorem}\label{t:msgal2}
Let $\lambda_h=u_h|_\btriH$, and $\lambda_h^{ms,\Pi}$ solution of~\eqref{e:acms-lod}. Then 
$\lambda_h-\lambda_h^{\rm{ms}}\in\tLambda^\triangle$ and
\[
|u_h^\Har-u_h^{ms,\Pi}|_{H_\A^1(\Omega)}^2
\le 9\alpha_{\rm{stab}}\HH^2\|g\|^2_{L^2_{\rho}(\Omega)}.
\]
\end{theorem}
\begin{proof}
First note that $\lambda_h-\lambda_h^{\rm{ms}}\in\tLambda^\triangle$ since it follows from the Galerkin orthogonality that $s(\lambda_h-\lambda_h^{ms,\Pi},\mu_h^{ms,\Pi})=0$ for all $\mu_h^{ms,\Pi}\in\Lambdams$. Using Lemma~\ref{l:WPI3} we obtain  
\begin{equation*}
|u_h^\Har-u_h^{ms,\Pi}|_{H_\A^1(\Omega)}^2
=\bigl(\rho g,T(\lambda_h-\lambda^{ms,\Pi})\bigr)
\le(9\alpha_{\rm{stab}})^{1/2}\HH\|g\|_{L_\rho^2(\Omega)}|T(\lambda_h-\lambda_h^{ms,\Pi})|_{H_\A^1(\Omega)}, 
\end{equation*}
and the result follows.
\end{proof}

\subsection{Decaying Results for High-Contrast coefficients} \label{ss:highcon}
We next prove that $P^{K,\triangle}\nu_h$ with $K\in\triH$ decay exponentially.
\begin{lemma}\label{l:decay2}
Let $\mu_h\in\Lambda_h$ and let $\tilde\phi_h^\triangle=P^{K,\triangle}\mu_h$ for some fixed element $K\in\triH$. Then, for any integer $j\ge1$, 
\begin{equation*}
|T\tilde\phi_h^\triangle|_{H_\A^1(\triH\backslash\T_{j+1}(K))}^2
\le 9\alpha_{\rm{stab}}|T\tilde\phi_h^\triangle|_{H_\A^1(\T_{j+2}(K)\backslash\T_j(K))}^2.
\end{equation*}
\end{lemma}
\begin{proof} 
Following the steps of the proof of Lemma~\ref{l:decay}, we gather that
\begin{equation*}
|T\tilde\phi_h^\triangle|_{H_\A^1(\triH\backslash\T_{j+1}(K))}^2
\le\sum_{\tau\in\T_{j+1}(K)\backslash\T_j(K)}|T\tnu|_{H_\A^1(\tau)}|T\tilde\phi_h^\triangle|_{H_\A^1(\tau)}, 
\end{equation*}
where $\tnu^\triangle\in\tLambda^\triangle$ is such that $\tnu^\triangle|_{\dtau}=\tilde\phi_h^\triangle$ if $\tau\in\triH\backslash\T_{j+1}(K)$, and $\tnu^\triangle=0$ on the remaining edges. If $e$ is an  edge of $\dtau$ and $\dtau'$, and $\chi_e$ the characteristic function of $e$, for $\tau\in\T_{j+1}(K)\backslash\T_j(K)$ and $\tau'\in\T_j(K)\backslash\T_{j-1}(K)$, then, for $\tmu^{e,\triangle}=\tmu^\triangle|_e$,
\[
|T(\tilde\mu_h^\triangle)  |^2_{H_\A^1(\tau)}
\le3\sum_{e \subset \partial \tau}|T(\chi_e\tilde\mu_h^\triangle)  |^2_{H_\A^1(\tau)}
\]
and 
\begin{multline*}
|T(\chi_e\tilde\mu_h^\triangle)  |^2_{H_\A^1(\tau)}
=(\tmu^{e,\triangle},S_{ee}^\tau\tilde\mu_h^{e,\triangle})_e
\le\alpha_{\rm{stab}}(\tilde\mu_h^{e,\triangle},(\widetilde S_{ee}^\tau+\widetilde S_{ee}^{\tau^\prime})\tilde\mu_h^{e,\triangle})_e
\\
\le\alpha_{\rm{stab}}\bigl((\tilde\mu_h^\triangle,S^\tau\tilde\mu_h^\triangle)_{\partial\tau}
+(\tilde\mu_h^\triangle,S^\tau\tilde\mu_h^\triangle)_{\partial\tau^\prime}\bigr)
=\alpha_{\rm{stab}}\bigl(|T\tmu^\triangle|_{H_\A^1(\tau)}^2+|T\tmu^\triangle|_{H_\A^1(\tau')}^2\bigr),
\end{multline*}
where we have used~\eqref{e:tildeSest}.  
\end{proof}
Note that now the bound is in terms of $\T_{j+2}(K)\backslash\T_j(K)$ rather than
$\T_{j+1}(K)\backslash\T_j(K)$. This means that the $j$ in Corollary~\ref{c:decay} is replaced below by the integer part of $(j+1)/2$. 

\begin{corollary}\label{c:decay2}
 Assume that $K\in\triH$ and $\nu_h\in\Lambda_h$ and let $\tilde{\phi}_h^\triangle=P^{K,\triangle} \nu_h\in\tLambda^\triangle$. Then, for any integer $j\ge1$, 
\begin{equation*}
|T\tilde{\phi}_h^\triangle|_{H_\A^1(\triH\backslash\T_{j+1}(K))}^2
\le e^{-\frac{[(j+1)/2]}{1+9\alpha_{\rm{stab}}}} |T\tilde{\phi}^\triangle_h|_{H_\A^1(\triH)}^2. 
\end{equation*}
where $[s]$ is the integer part of $s$. 
\end{corollary}
\begin{proof} 
Using Lemma~\ref{l:decay2} we have 
\begin{multline*}
|T\tilde{\phi}_h^\triangle|_{H_\A^1(\triH\backslash{\T}_{j+1}(K))}^2
\le|T\tilde{\phi}_h^\triangle|_{H_\A^1(\triH\backslash{\T}_j(K))}^2
\\
\le9\alpha_{\rm{stab}}|T\tilde{\phi}_h^\triangle|_{H_\A^1(\triH\backslash\T_{j-1}(K))}^2
-9\alpha_{\rm{stab}}|T\tilde{\phi}_h^\triangle|_{H_\A^1(\triH\backslash{\T}_{j+1}(K))}^2, 
\end{multline*}
and then
\[
|T\tilde\phi_h^\triangle|_{H_\A^1(\triH\backslash\T_{j+1}(K))}^2
\le\frac{9\alpha_{\rm{stab}}}{1+9\alpha_{\rm{stab}}}|T\tilde{\phi}_h^\triangle|_{H_\A^1(\triH\backslash\T_{j-1}(K))}^2
\le e^{-\frac1{1+9\alpha_{\rm{stab}}}} |T\tilde{\phi}_h^\triangle|_{H_\A^1(\triH\backslash\T_{j-1}(K))}^2. 
\]
\end{proof} 

Inspired by the exponential decay stated in Corollary~\ref{c:decay2}, we define the operator $P^{\triangle,j}$ as follows. First, for a fixed $K\in\triH$, let
\[
\tLambda^{\triangle,K,j}=\{\tmu\in\tLambda^\triangle:\,T\tmu=0\text{ on }\triH\backslash\T_j(K)\}. 
\]
For $\mu_h\in\Lambda_h$, define $P^{\triangle,K,j}\mu_h\in\tLambda^{K,j}$ such that 
\[
s(P^{\triangle,K,j}\mu_h,\tmu)
=s_K(\mu_h,\tmu)\quad\text{for all }\tmu\in\tLambda^{\triangle,K,j}, 
\]
and let
\begin{equation}\label{e:Pjdeltadef}
P^{\triangle,j}\mu_h=\sum_{K\in\triH}P^{\triangle,K,j}\mu_h. 
\end{equation}

Finally, define the approximation $\lambda_H^{\Pi,j}\in\Lambda_H^\Pi$ such that
\begin{equation}\label{e:lambdamspijdef}
s\bigl((I-P^{\triangle,j})\lambda_H^{\Pi,j},(I-P^{\triangle,j})\mu_H^\Pi\bigr)
=(\rho g,T(I-P^{\triangle,j})\mu_H^\Pi)\quad\text{for all }\mu_H^\Pi\in\Lambda_H^\Pi, \end{equation}
and then let $\lambda_h^{ms,\Pi,j}=(I-P^{\triangle,j})\lambda_H^{\Pi,j}$ and $u^{ms,\Pi,j}_h=T\lambda_h^{ms,\Pi,j}$. We name this the ACMS--LSD (\emph{Approximated Component Mode Synthesis Localized Spectral Decomposition}) method. 

We now analyze the approximation error of the method, starting by a technical result essential to obtain the final estimate. 

\begin{lemma}\label{l:errorPdeltaj}
  Consider $\nu_h\in\Lambda_h$ and the operators $P^\triangle$ defined by
  ~\eqref{e:Pdelta} and $P^{\triangle,j}$ by~\eqref{e:Pjdeltadef} for $j>1$. Then
\begin{equation*}
|T(P^\triangle-P^{\triangle,j})\nu_h|_{H_\A^1(\triH)}^2
\le (c_\gamma j)^2(9\alpha_{\rm{stab}})^2e^{-\frac{[(j-1)/2]}{1+9\alpha_{\rm{stab}}}}|T\nu_h|^2_{H_\A^1(\triH)}, 
\end{equation*}
where $c_\gamma$ is as in~\eqref{e:cgammadef}. 
\end{lemma}
\begin{proof}
Let $\tilde\psi_h^\triangle=(P^\triangle-P^{\triangle,j})\nu_h=\sum_{K\in\triH}(P^{K,\triangle}-P^{K,\triangle,j})\nu_h$. For each $K\in\triH$, let $\tilde\psi_h^{K,\triangle}\in\widetilde\Lambda_h^\triangle$ be such that $\tilde\psi_h^{K,\triangle}|_e=0$ if $e$ is an edge  of an element of $\T_j(K)$ and $\tilde\psi_h^{K,\triangle}|_e=\tilde\psi_h^\triangle|_e$, otherwise. We obtain 
\begin{equation}\label{e:splitting2}
|T\tilde\psi_h^\triangle|_{H_\A^1(\triH)}^2
=\sum_{K\in\triH}\sum_{\tau\in\triH} 
s_\tau(\tilde\psi_h^\triangle-\tilde\psi_h^{K,\triangle},(P^K-P^{K,\triangle,j})\nu_h)
+s_\tau(\tilde\psi_h^{\triangle,K},(P^{K,\triangle}-P^{K,\triangle,j})\nu_h). 
\end{equation}
See that the second term of~\eqref{e:splitting2} vanishes since 
\[
\sum_{\tau\in\triH}s_\tau(\tilde\psi^{K,\triangle},(P^{K,\triangle}-P^{K,\triangle,j})\nu_h)_\dtau
=\sum_{\tau\in\triH}s_\tau(\tilde\psi^{K,\triangle},P^{K,\triangle}\nu_h)_\dtau=0.
\]
For the first term of~\eqref{e:splitting}, as in Lemma~\ref{l:decay}, 
\begin{multline*}
\sum_{\tau\in\triH}s_\tau(\tilde\psi_h^\triangle-\tilde\psi_h^{K,\triangle},(P^{K,\triangle}-P^{K,\triangle,j})\nu_h)_{\partial\tau}
\\
\le\sum_{\tau\in\T_{j+1}(K)}
|T(\tilde\psi_h^\triangle-\tilde\psi_h^{K,\triangle})|_{H_\A^1(\tau)}
|T(P^{K,\triangle}-P^{K,\triangle,j})\nu_h|_{H_\A^1(\tau)}
\\
\le3\alpha_{\rm{stab}}^{1/2}|T\tilde\psi_h^\triangle|_{H_\A^1(\T_{j+1}(K))}|T(P^{K,\triangle}-P^{K,\triangle,j})\nu_h|_{H_\A^1(\T_{j+1}(K))}.
\end{multline*}
Let $\nu_h^{K,\triangle,j}\in\tLambda^{K,\triangle,j}$ be equal to zero on all faces of elements of $\T_H\backslash\T_j(K)$ and equal to $P^{K,\triangle}\nu_h$ otherwise. Using Galerkin best approximation property,  Lemma~\ref{l:half_00} and Corollary~\ref{c:decay2}, we obtain
\begin{multline*}
|T(P^{K,\triangle}-P^{K,\triangle,j})\nu_h|_{H_\A^1({\T}_{j+1}(K))}^2
\le|T(P^{K,\triangle}-P^{K,\triangle,j})\nu_h|_{H_\A^1(\triH)}^2
\\ 
\le|T(P^{K,\triangle}\nu_h-\nu_h^{K,\triangle,j})|_{H_\A^1(\triH)}^2
\le9\alpha_{\rm{stab}}|TP^{K,\triangle}\nu_h|_{H_\A^1(\triH\backslash\T_{j-1}(K))}^2 
\\
\le9\alpha_{\rm{stab}}e^{-\frac{[(j-1)/2]}{1+9\alpha_{\rm{stab}}}}|TP^{K,\triangle}\nu_h|_{H_\A^1(\triH)}^2.
\end{multline*}
We gather the above results to obtain  
\begin{multline*}
|T\tilde\psi_h^\triangle|_{H_\A^1(\triH)}^2
\le9\alpha_{\rm{stab}} e^{-\frac{[(j-1)/2]}{2(1+2\alpha_{\rm{stab}})}}
\sum_{K\in\triH}|T\tilde\psi_h^\triangle|_{H_\A^1(\T_{j+1}(K))}|TP^{K,\triangle}\nu_h|_{H_\A^1(\triH)}
\\
\le 9\alpha_{\rm{stab}}e^{-\frac{[(j-1)/2]}{2(1+9\alpha_{\rm{stab}})}} c_\gamma j|T\tilde\psi_h^\triangle|_{H_\A^1(\triH)}
\biggl(\sum_{K\in\triH}|TP^{K,\triangle}\nu_h|_{H_\A^1(\triH)}^2\biggr)^{1/2}. 
\end{multline*}

We finally gather that 
\begin{multline*}
|TP^{K,\triangle}\nu_h|_{H_\A^1(\triH)}^2
=s(P^{K,\triangle}\nu_h,P^{K,\triangle}\nu_h)_{\partial\triH}
=s_K(P^{K,\triangle}\nu_h,\nu_h)
\\
=\int_K\A\bgrad(TP^{K,\triangle}\nu_h)\cdot\bgrad T\nu_h\,d\xx, 
\end{multline*}
and from Cauchy--Schwarz, $|TP^{K,\triangle}\nu_h|_{H_\A^1(\triH)}\le|T\nu_h|_{H_\A^1(K)}$,
we have
\[
\sum_{K\in\triH}|TP^{K,\triangle}\nu_h|_{H_\A^1(\triH)}^2\le|T\nu_h|_{H_\A^1(\triH)}^2.
\]
\end{proof}

\begin{theorem}\label{t:highconterror}
Define $u_h^\Har$ by~\eqref{e:splitpdes} and let $u^{ms,\Pi,j}_h=T(I-P^{\triangle,j})\lambda_H^{\Pi,j}$, where $\lambda_H^{\Pi,j}$ is as in~\eqref{e:lambdamspijdef}. Then
\[
|u_h^\Har-u_h^{ms,\Pi,j}|_{H_\A^1(\triH)}
\le \HH
\biggl(3(\alpha_{\rm{stab}})^{1/2}+c_\gamma j 9\alpha_{\rm{stab}}e^{-\left(\frac{[(j-1)/2]}{2(1+9\alpha_{\rm{stab}})} -\log(c_{P,G}/\HH)\right) }\biggr)\|g\|_{L_\rho^2(\Omega)}. 
\]
\end{theorem}
\begin{proof}
First, from the triangle inequality, 
\[
|u_h^\Har-u^{ms,\Pi,j}_h|_{H_\A^1(\triH)}
\le|u_h^\Har-u^{ms,\Pi}_h|_{H_\A^1(\triH)}+|u_h^{ms,\Pi}-u^{ms,\Pi,j}_h|_{H_\A^1(\triH)}, 
\]
and for the first term we use Theorem~\ref{t:msgal2}. For the second term, we define $\hat u^{ms,\Pi,j}_h=T(I-P^{\triangle,j})\lambda_h^{ms,\Pi}$, and then
\[
u^{\rm{ms}}_h-\hat u^{ms,\Pi,j}_h
=T(P-P^j)\lambda_h^{ms,\Pi}. 
\]
Relying on the Galerkin best approximation we gather from Lemma~\ref{l:errorPdeltaj} that
\[
|u_h^{\rm{ms}}-u_h^{ms,\Pi,j}|_{H_\A^1(\triH)}^2
\le|u_h^{\rm{ms}}-\hat u^{ms,\Pi,j}_h|_{H_\A^1(\triH)}^2
\le (c_\gamma j)^2(9\alpha_{\rm{stab}})^2e^{-\frac{[(j-1)/2]}{(1+9\alpha_{\rm{stab}})}}|T\lambda_h^{ms,\Pi}|_{H_\A^1(\triH)}^2.
\]
Since $u^{\rm{ms}}_h=T\lambda^{\rm{ms}}_h$ the result follow from
Lemma~\ref{l:WPI3} and the global Poincar\'e's inequality~\eqref{e:umsest}. 
\end{proof}

\begin{rem} The localization required for the a priori error estimate of the previous theorem depends just on the logarithm of the global Poincar\'e's inequality constant $c_{P,G}$ defined in~\eqref{e:umsest}. Furthermore, it is common in the literature to assume the global condition  $\amin\geq1$ and $\rho=1$, which implies $c_{P,G} = O(1)$. We can weaken this condition  by choosing $\rho(\xx) =  a_-(\xx)$ for almost all $\xx \in \Omega$ so that when the weighted Poincar\'e inequality holds~\cite{MR3047947,zbMATH06125818} then $c_{P,G} = O(1)$; a simple example of such a case is when $\Omega$ is composed of several inclusions with small permeabilities and surrounded by a material with a large permeability.
\end{rem}

\begin{rem}
  The complexity of the proposed method depends on the number of 
  eigenvalues of (\ref{e:eigenvalue}) below a threshold $\alpha_{stab}$
  associated to each edge $e$. There are several works on the literature
  discussing this issue for similar eigenvalues problems~\cite{MR2728702,MR4320895,MR4683905} and the number is related to
  the amount of channels of high permeability crossing the edge $e$.
  For the numerical experiments tested in~\cite{zbMATH07621527}, just one
  eigenvalue per edge was enough for $\HH = H$, $\HH = H/2$ and $\HH = H/4$ 
  with $\alpha_{\rm{stab}} = 1.5$ 
  \end{rem} 

\section{Spectral Multiscale Problems inside Substructures} \label{s:localproblems} 
Recall the decomposition $u_h=u_h^B+u_h^\Har$, and so far we derived a scheme that approximates $u_h^\Har$ only. From~\eqref{e:splitpdes}, we gather that $u_h^B$ is defined locally. Fixing an element $\tau \in \triH$, we
introduce a multiscale method by first building the approximation 
space $V_\tau^{\rm{ms}}:=\span\{\psi_h^1,\psi_h^2,\cdots,\psi_h^{N_\tau}\}$
generated by the following generalized eigenvalue problem: Find the eigenpairs
$(\alpha_i,\psi^i_{h})\in(\RR,V_h^B(\tau))$ such that 
\[
a_\tau(v_h, \psi^i_h) = \lambda_i (\rho v_h, \psi^i_h)_\tau   \qquad \text{for all}
\quad  v_h \in V_h^B(\tau)
\]
where
\[
a_\tau(v_h, \psi^i_h) = \int_\tau\A\bgrad v_h \cdot\bgrad \psi^i_h\,d\xx\qquad
\text{and} \qquad (\rho v_h, \psi^i_h)_\tau = \int_\tau \rho v_h  \psi^i_h\,d\xx
\]
and $0 < \lambda_1 \leq \lambda_2 \leq \cdots \lambda_{N_\tau} < 1/{\HH^2}$
  and $\lambda_{N_\tau+1} \geq 1/{\HH^2}$. Again, $\HH$ is the user defined target accuracy target. For instance $\HH = H$ or $\HH = h^r, 0<r\leq 1$. The local multiscale problem is defined by: Find $u_h^{B,\rm{ms}} \in V_h^{\rm{ms}}$ such that
\[
a_\tau(u^{B,\rm{ms}}_h,v_h) = (\rho g, v_h)_\tau \qquad \text{for all } v_h \in V_h^{B,\rm{ms}}.
\]
It then follows that 
\[
|u^B_h - u_h^{B,\rm{ms}}|_{H^1_\A(\tau)}^2 = (\rho g,u^B_h - u^{B,\rm{ms}}_h)_\tau
\leq \HH |u^B_h - u^{B,\rm{ms}}_h|_{H^1_\A(\tau)} \|g\|_{L^2_\rho(\tau)}
\]
and $|u^B_h - u^{B,\rm{ms}}|_{H^1_\A(\tau)} \leq\HH\|g\|_{L^2_\rho(\tau)}$. See~\cite[Section 4]{MR4238010} for a similar computation. 

\section{Numerical Experiment Tests}
Let $\Omega = (0,1)^2$ and consider a Cartesian mesh of $2^M \times 2^M$ coarse squares, and then subdivide each square into $2^{N-M} \times 2^{N-M}$ equal fine squares and then subdivide further into two 45-45-90 triangular elements. Denote $H = 2^{-M}$ and $h=2^{-N}$ as the sizes of the coarse and the fine elements, respectively. We remind that $\alpha_{\rm{stab}}>1$, introduced in (\ref{e:Lpmdef}), controls how fast the multiscale basis functions decays  exponentially (see Corollary \ref{c:decay2}), while $\HH$ is the target precision of the method (see Theorem~\ref{t:msgal2}), and was introduced in the definition of $\hat{S}^\tau_{ee}$ to build the generalized eigenvalue problem (\ref{e:eigenvalue}). To obtain an effective decay we should choose $\alpha_{\rm{stab}}$ close to one, while keeping the number of eigenvalues above $\alpha_{\rm{stab}}$ small. We assume the bubble solution
$u_h^B$ is computed exactly, therefore, $u_h^\Har - u_h^{ms,\Pi} = u_h - u_h^{\rm{ms}}$ and $u_h^\Har - u_h^{ms,\Pi,j} = u_h - u_h^{\rm{ms},j}$ where $u_h^{\rm{ms},j} = u_h^B +u_h^{\rm{ms},j}$.

For the first numerical test, we find a judicious choice of $\alpha_{\rm{stab}}$.
We consider $\A(\xx) = 1$ and $\rho(\xx) = 1$  and $\HH = H$ for different values of $H/h$. Table \ref{tab1} shows the seven largest generalized eigenvalues of (\ref{e:eigenvalue}). We consider the case where the coarse edge $e$ of $\btriH\backslash\partial\Omega$ is shared by the coarse elements $\tau$ and $\tau^\prime$ of $\triH$ where $\tau$ and $\tau^\prime$ do not share an edge with $\partial \Omega$. This is the coarse edge case which the eigenvalues are the largest. We choose $\alpha_{\rm{stab}}$ so that we can remove the weak corner singularities. The choices $\alpha_{max}=1.5$ and $1.3$ are good compromises between the fast decay of the  multiscale basis functions while keeping the number of selected large  eigenvalues small. When $A(\xx)$ is highly heterogeneous or has high-contrast coefficients channels crossing the coarse edge $e$, very large eigenvalues might appear on the generalized eigenvalue problem~\eqref{e:eigenvalue}, and for $\alpha_{\rm{stab}}= 1.5$ or $1.3$ these large eigenvalues will also be selected. See in Table~\ref{tab1} that the smallest eigenvalues converge fast to one. 

\begin{table}
\begin{center}
%{\scriptsize
\begin{tabular}{|c|c|c|c|c|c|c|c|}\hline
  $H/h$ & ~$\alpha_{1}^e$  & ~$\alpha_{2}^e$  & ~$\alpha_{3}^e$  & ~$\alpha_{4}^e$ & ~$\alpha_{5}^e$ &  ~$\alpha_{6}^e$ &  ~$\alpha_{7}^e$\\ \hline
  $8$   &1.7196&1.1503&1.0134&1.0013&1.0001&1.0000&1.0000 \\ \hline
  $16$  &2.2140&1.3272&1.0525&1.0109&1.0015&1.0002&1.0000    \\ \hline
  $32$  &2.8065&1.5558&1.1199&1.0354&1.0078&1.0019&1.0004  \\  \hline
 
\end{tabular}\vskip3mm
%}
\caption{The 7th largest eigenvalues of the eigenvalue problem
  (\ref{e:eigenvalue}) with $\HH = H$.}
\label{tab1}
\end{center}
\end{table}

The second numerical test is borrowed from~\cite{zbMATH07621527}. We consider the distribution of $\A(\xx)$ to be a channel on the form of a ${\bf{H}}$. We examine the exponential decay of the multiscale basis functions under the presence of large coefficient on the ${\bf{H}}$-shape region. 

We assume that $\A(\xx)$ is scalar and that $\rho(\xx) = \A(\xx)$. The distribution of $\A(\xx)$ is shown on the left of Figure~\ref{fig1}, where $\A = 100$ inside the ${\bf{H}}$-shape region and $A\ = 1$ outside. Let $N = 6$ and $M = 3$. This distribution of the coefficients $\A$ has the property that $\A(1/2,1/2)=100$ at the
  corner coarse node $\xx =(1/2,1/2)$ and $\A =1$ at the remaining corner
  coarse nodes. The right Figure~\ref{fig1} shows the (lack of) decay (in the log-normal scale) of the multiscale basis function associated to the coarse
  node $\xx=(1/2,1/2)$ without the eigenfunctions. We can see
  that this multiscale basis function does not decay fast 
  away from $\A(1/2,1/2)$ on the ${\bf{H}}$-shape region. The presence of the white holes
  in this picture is because the value of the function is closed to zero since this multiscale basis function vanishes at all coarse nodes but at $(1/2,1/2)$. 
  The reason for this non-decay is due to that this multiscale  basis function
  ``wants'' to have small energy, therefore, ``wants'' to be near one
  on the ${\bf{H}}$-shape region, where $\A$ is large. We now consider the
  adaptive case with $\alpha_{\rm{stab}} = 1.5$. Figure~\ref{fig2} shows the exponential decay  when $\HH = H$, as predicted by the theory.

\begin{figure}
\centering
\epsfig{file=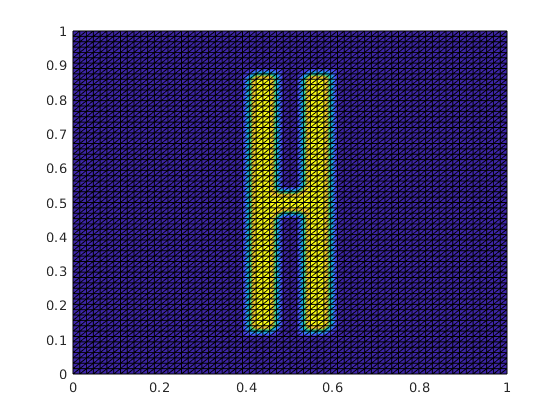, scale=0.55}
\epsfig{file=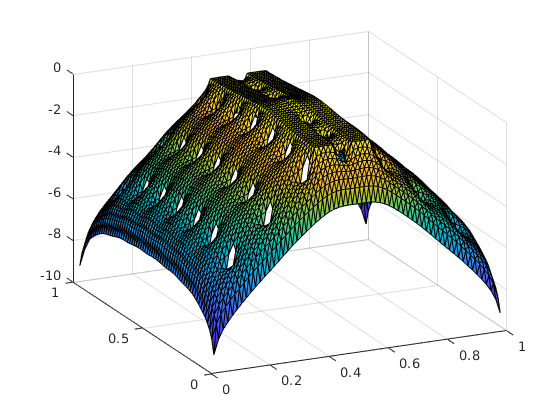, scale=0.55}
  \caption{On the left, the distribution of the coefficient for a
          $8\times 8$ subdomain decomposition. On the right, the
          log-normal plot of the multiscale basis functions without adaptivity. Note that there is no exponential decay whatsoever.} 
\label{fig1}
\end{figure}

\begin{figure}
\centering
\epsfig{file=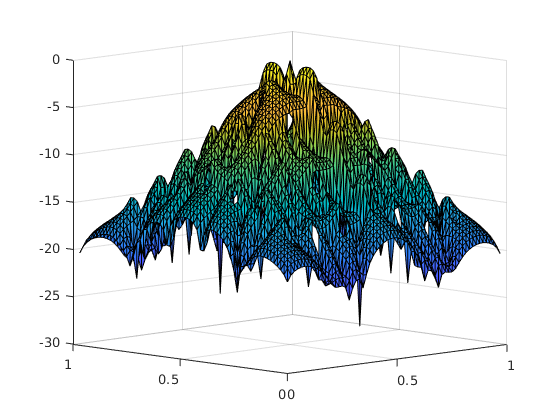, scale=0.6}
  \caption{Log-normal plot showing the  decay of a multiscale basis functions with adaptivity, for $\HH =H$.} 
\label{fig2}
\end{figure}

In the third numerical experiment we keep the same distribution of coefficients
in Figure~\ref{fig1} with $N = 6$ and $M = 3$. To make the problem
a little more complicated, we multiply $\A$ and $\rho$ in each element
by independently uniform random distributions between zero and one.
Similarly, we let $f$ to be constant in each element given by another
independently uniformly random distributions between zero and one.
Table~\ref{tab2} shows the energy errors for different values of $\HH$.
We also include the total number of edges functions required for a
$\HH$ tolerance. We take $\alpha_{\rm{stab}} = 1.5$. We note that there are $49$ interior coarse corners nodes, $112$ interior coarse edges, and $7$ fine nodes per open coarse edges. By taking
$\HH = H$, the number of edge multiscale basis functions is smaller than the
number of coarse edges. 

In Table~\ref{tab3} we provide the energy error for different values of
$\alpha_{\rm{stab}}$; by adding few edge multiscale basis functions on the
approximation space, we see a dramatic improvement on the energy error.
In Table~\ref{tab4} we refine the mesh and consider
$H=1/16$, $H/h = 16$, $\HH= H$ and compute the energy error for different values of localization $j$; here, $j=-1$ means the classical Adaptive BDDC
\cite{MR3350292}; also See that a small localization $j$ is sufficient for $\alpha_{\rm{stab}}$ sufficiently small.

\begin{table}
\begin{center}
%{\scriptsize
\begin{tabular}{|c|c|c|c|c|}\hline
  $\HH$ & $|u_h - u^{\rm{ms}}_h|_{H^1_\A}$  & $\frac{|u_h - u^{\rm{ms}}_h|_{H^1_\A}}{|u_h|_{H^1_\A}}$
  & $\frac{|u_h - u^{\rm{ms}}_h|_{H^1_a}}{\|f\|_{L^2_\rho}}$  & Neigs \\ \hline
  1/8   &  0.0095     &   0.0083   &    0.0079 &  78  \\
  \hline
  1/16   &  0.0064     &   0.0056   &    0.0053  &   92      \\
  \hline
  1/32   &  0.0025    &   0.0022    &    0.0021  &   112   \\
  \hline
  1/64   &  0.0014   & 0.0012  & 0.0011 & 226 \\
  \hline
\end{tabular}\vskip3mm
\caption{The energy errors for different accuracy targets $\HH$. The last column shows Neigs (the total number of multiscale edges functions). We let $\alpha_{\rm{stab}} = 1.5$.}

%}
\label{tab2}
\end{center}
\end{table}

\begin{table}    
\begin{center}
%{\scriptsize
\begin{tabular}{|c|c|c|}\hline
 $\alpha_{\rm{stab}}$ &  $|u_h - u^{\rm{ms}}_h|_{H^1_\A}$ & Neigs  \\ \hline
  100 & 0.299   &   0  \\ \hline
  10 & 0.0964 & 8 \\ \hline
  2 & 0.0204 & 14 \\ \hline
  1.5 & 0.0107 &  79 \\ \hline 
  1.3 & 0.0084 &  103 \\ \hline
  1.2 & 0.0075 & 112\\ \hline
  1.1 & 0.0063 & 171 \\ \hline
  \end{tabular}\vskip3mm
%}
\caption{The energy errors $\HH=H$ and different values of $\alpha_{\rm{stab}}$.}
\label{tab3}
\end{center}
\end{table}

\begin{table}
  \begin{center}
%{\scriptsize
\begin{tabular}{|c|c|c|c|c|c|}\hline
  $\alpha_{\rm{stab}}(Neigs) \backslash j$ & -1 &  0 & 1 & 2  & 3 \\
  \hline
  10(24)   & 0.6913 & 0.6543   & 0.2562   & 0.0547  & 0.0305  \\
  \hline 
  3 (55)   & 0.6797 & 0.5236   & 0.2149   & 0.0378  & 0.0154 \\
  \hline
  2(416)   & 0.6562 & 0.1225   & 0.0056   & 0.0057  &  0.0057    \\
  \hline
  1.1(1286)& 0.5908 & 0.0031   & 0.0022   & 0.0022  & 0.0022   \\
  \hline
 \end{tabular}\vskip3mm
%}
\caption{The energy errors $u_h - u_h^{\rm{ms},j}$ with $\HH=H, H=1/16, H/h=1/16$ and different values of $\alpha_{\rm{stab}}$ and localization $j$.}
\label{tab4}
\end{center}
\end{table}

%%% TeX-master: "hybridlod19"
%%% End: 
\bibliographystyle{plain}
\bibliography{referencias}

\end{document}